\documentclass[a4paper,12pt]{amsart}

\usepackage{amsmath}
\usepackage{amssymb,amsbsy,amsmath,amsfonts,amssymb,amscd}
\usepackage{latexsym}
\usepackage{amsthm}
\usepackage{graphics}
\usepackage{color}
\usepackage{tikz}
 \usepackage{tikz-3dplot}
\usetikzlibrary{arrows}
\usepackage{tikz-cd}
\usepackage{pdfpages} 
\usepackage{longtable}
\usepackage{xy}
\usepackage{array}
\usepackage{color}
\usepackage{comment}

\input xy
\xyoption{all}

\newcommand\sC{{\mathcal C}}

\newcommand\sA{{\mathcal A}}
\newcommand\sF{{\mathcal F}}

\newcommand\sQ{{\mathcal Q}}

\newcommand\sB{{\mathcal B}}
\newcommand\sP{{\mathcal P}}

\usepackage{textcomp} 
               
\newcommand\la{\lambda}

\newcommand\be{\beta}
\newcommand\e{\epsilon}
\newcommand\s{\sigma}

\newcommand\Ga{\Gamma}

\newcommand\ga{\gamma}
\newcommand\de{\delta}

\newcommand{\NN}{\ensuremath{\mathbb{N}}}
\newcommand{\hol}{\ensuremath{\mathcal{O}}}

\newcommand{\PP}{\ensuremath{\mathbb{P}}}

\newcommand{\FF}{\ensuremath{\mathbb{F}}}

\newcommand{\ra}{\ensuremath{\rightarrow}}

\def\eea{\end{eqnarray*}}
\def\bea{\begin{eqnarray*}}

\newcommand\dual{\mathrel{\raise3pt\hbox{$\underline{\mathrm{\thinspace d
\thinspace}}$}}}
\newcommand\qe{\ifhmode\unskip\nobreak\fi\quad $\Box$}       

\def\BOX{\hfill\lower.5\baselineskip\hbox{$\Box$}}

\newtheorem{theorem}{Theorem}
\newtheorem{theo}[theorem]{Theorem}
\newtheorem{rem}[theorem]{Remark}

\newtheorem{prop}[theorem]{Proposition}
\newtheorem{cor}[theorem]{Corollary}
\newtheorem{lemma}[theorem]{Lemma}
\newtheorem{example}[theorem]{Example}
\newenvironment{ex}{\begin{example}\rm}{\end{example}}

\theoremstyle{definition}
\newtheorem{defin}[theorem]{Definition}

\newcommand{\sR}{\ensuremath{\mathcal{R}}}

\newenvironment{dedication}
        {\begin{quotation}\begin{center}\begin{em}}
        {\par\end{em}\end{center}\end{quotation}}

\usepackage{hyperref}

\makeatletter
\def\tagform@#1{\maketag@@@{\ignorespaces#1\unskip\@@italiccorr}}
\makeatother

\newcolumntype{H}{@{}>{\lrbox0}l<{\endlrbox}} 

\begin{document}

\title[Singularities of quartic surfaces]{Singularities  of  normal quartic  surfaces I (char=2).}
\author{Fabrizio Catanese}
\address{Lehrstuhl Mathematik VIII, 
 Mathematisches Institut der Universit\"{a}t
Bayreuth, NW II\\ Universit\"{a}tsstr. 30,
95447 Bayreuth, Germany \\ and Korea Institute for Advanced Study, Hoegiro 87, Seoul, 
133--722.}
\email{Fabrizio.Catanese@uni-bayreuth.de}

\thanks{
The author acknowledges support of the ERC 2013 Advanced Research Grant - 340258 - TADMICAMT}

\maketitle

\begin{dedication}
Dedicated to Bernd Sturmfels on the occasion of his 60-th  
 birthday.
\end{dedication}

\begin{abstract}
We show, in this first part,  that the maximal number of singular
points of a  normal quartic surface $X \subset \PP^3_K$ defined over an algebraically closed field $K$ of characteristic $2$ is at most $16$.  
We produce examples with $14$,  respectively $12$, singular points and show that,  under several geometric assumptions ($\mathfrak S_4$-symmetry,
or behaviour of the Gauss map,  or structure of  tangent cone at  one of the   singular points $P$, separability/inseparability of the projection
with centre $P$), we can obtain   smaller upper bounds  for the number of singular points of $X$.

\end{abstract}

\tableofcontents

\setcounter{section}{0}

\section{Introduction} 

Given an irreducible surface $X$ of degree $d$ in $\PP^3$, defined over an algebraically closed  field $K$, which is normal, that is, with only   finitely many singular points, one important question is to determine 
 the maximal number $\mu(d)$  of singular points that $X$ can have (observe however, see for instance \cite{bsegre}, \cite{angers}, \cite{j-r}, \cite{p-t}, \cite{duco},
 that the research has been more focused on the seemingly simpler question of finding the maximal number of nodes, that is, ordinary quadratic singularities). The case of $d=1,2$ being trivial ($\mu (1)= 0, \mu (2)=1$), the first interesting cases are for $d= 3,4$. 
 
 We have that $\mu (3)= 4$,  while $ \mu (4)=16$  if char(K) $ \neq 2$.

\bigskip

For $d=3$ (see  proposition \ref{monoid} and for instance \cite{kummers} for  more of classical references), a normal cubic surface $X$ can have at most 4 singular points, no three of them can be collinear, 
and if it does have 4 singular points, these are linearly independent, hence $X$  is projectively equivalent to the 
so-called Cayley cubic, first apparently found by Schl\"afli, see \cite{schlaefli}, \cite{cremona},  \cite{cayley}
(and then the singular points are nodes).

The Cayley cubic has the simple equation

$$X : = \{ x : = (x_0, x_1, x_2, x_3) | \s_3(x ): =  \sum_i \frac{1}{x_i} x_0 x_1 x_2 x_3 = 0\} $$ 
Here $\s_3$ is the third elementary symmetric function (the four singular points are the 4 coordinate points).

 Even the classical case of cubic surfaces still offers plenty of open questions
\cite{Ra-Stu}: in this article we go up to
the case of quartic surfaces.

\bigskip
 
  The main purpose of this paper  is indeed to show  with elementary methods
   (Theorem \ref{mt}) that, if $ char(K) = 2$, then  $\mu(4)\leq  16$,
   and to provide easy examples which lead to the
    conjecture that $\mu(4)\leq  14$, which will be proven in Part II of this article,
    in cooperation with Matthias Sch\"utt (using the special features of elliptic fibrations in characteristic $2$).
  
  \bigskip
  
 Whereas  symmetric functions produce surfaces with the maximal number of singularities for degree $d=3$,
or for $d=4$ in characteristic $\neq 2$ (see for instance \cite{kummers}, \cite{mukai}) we show 
in the last section  that for $d=4$ and 
char $=2$ symmetric functions produce quartics with  at most  $12$ singular points,
and explicit examples with $0,1,4,5,6,10,12$ singular points.

In this paper we  produce the following explicit  example, of  quartic surfaces with $14$ singular points \footnote{and one more in subsection 2.3, suggested by Matthias Sch\"utt.},
and producing what we call  `the inseparable case':
$$X : = \{  (z, x_1, x_2, x_3) |  z^2 (x_1 x_2 + x_3^2) + (y_3 + x_1)(y_3 + x_2) y_3 (y_3 + x_1+ x_2)=0 ,$$
$$ y_3 : =  a_3 x_3 + a_1 x_1 + a_2 x_2, \  a_3 \neq 0, a_1, a_2, a_3  \ {\rm  general} \} .$$

A normal  quartic surface can have,   if char(K) $ \neq 2$, at most 16 singular points.
Indeed,  if char(K) $ \neq 2$, and $X$ is a normal quartic surface, by proposition \ref{monoid} it has at most $7$ singular points if it has a triple point,
else  it suffices to project from a double  point of the quartic  to the plane, and to use the bound for the number of singular points for a plane curve of degree $6$, which equals $15$, to establish that $X$ has at most $16$ singular points. 

Quartics with $16$ singular points  (char(K) $ \neq 2$) have necessarily nodes as singularities, and they are the so called 
 Kummer surfaces \cite{kummer} (the first examples were found by Fresnel, 1822).

There is a long history of research on Kummer quartic surfaces in char(K) $ \neq 2$, for instance it is  
 well known  that if $d=4, \mu = 16$, then $X$ is the quotient of a
principally polarized  Abelian surface
$A$ by the group $\{\pm 1\}$.

But in char $=2$ \cite{shioda} Kummer surfaces behave differently, and have at most $4$ singular points  \footnote{ A  similar construction, introduced in \cite{kon-schr}, and based on other group schemes, leads to `Kummer' surfaces $X$ with 17 singular points: these,
by the results of this paper, cannot be isomorphic to quartic surfaces in $\PP^3$.}.

In this paper we show among other results   that, if $X$ is a normal quartic surface defined over an  algebraically closed field $K$ of characteristic $2$, then:
\begin{enumerate}
\item
If $X$ has a point of multiplicity $3$, then $ | Sing(X)| \leq 7$. (Proposition \ref{monoid}).
\item
If $X$ has a point of multiplicity $2$ such that the projection with centre $P$ is inseparable, then $ | Sing(X)| \leq 16$ (Proposition \ref{inseparable}, see steps I) and II)  for smaller upper bounds under special assumptions).
\item
$ | Sing(X)| \leq  16$ (Theorem \ref{mt}), and equality holds only if the singularities are all nodes ($A_1$-singularities, double points with smooth projective tangent cone) or 15 nodes and an $A_2$-singularity.
\end{enumerate}

This paper has a  big overlapping with  the previous preprint \cite{quartics}, improves the main result  and most results in loc. cit.,
and   supersedes it.

In Part II, in cooperation with Matthias Sch\"utt \cite{part2} we  show: 

i)  the better upper bound  $ | Sing(X)| \leq 14$
holds,  with  equality  only if the singular points are nodes and the minimal resolution of $X$
is a supersingular K3 surface. 

ii) if the Gauss image is a plane, equivalently
if the equation contains  only even powers of one of the variables,
then, counting multiplicities,   $ | Sing(X)| \geq 14$ and generically the surface $X$  has 14 nodes as singularities;
one may ask whether   conversely all quartics with $14$ nodes arise in this way.

\subsection{Notation and Preliminaries} For a point in projective space, we shall freely use  the vector notation $(a_1, \dots, a_{n+1})$,
instead of the more precise  notation $[a_1, \dots, a_{n+1}]$, which denotes the equivalence class of the above vector.

Let $Q(x_1,x_2,x_3)=0$ be a conic over a field $K$ of characteristic $2$.

Then we can write  $$Q(x_1,x_2,x_3)= \sum_i b_i^2 x_i^2 + \sum_{i<j} a_{ij} x_i x_j= ( \sum_i b_i x_i)^2 + \sum_{i<j} a_{ij} x_i x_j.$$
One finds that, unless $Q$ is the square of a linear form,  $[a] : = (a_{23},a_{13},a_{12})$ is the only point where the gradient of $Q$ vanishes. 

Taking coordinates such that $[a] = (0,0,1)$ we have that $Q(x) = x_1 x_2 + b(x)^2$, where $b(x)$ is a
(new)  linear form: for instance, if $Q(x_1,x_2,x_3)= x_1 x_2 + x_1 x_3 + x_2x_3$, then $Q  = (x_1 + x_2)(x_1 + x_3) + x_1^2$.

We have two cases  $Q[a]=0$, hence $b_3=0$, hence $Q(x) = x_1 x_2 + b(x_1, x_2)^2$, hence changing again coordinates
we reach the normal form $Q = x_1 x_2$; while if $b_3 \neq 0$ we reach the normal form  $Q = x_1 x_2 + x_3^2$.

Hence we have just the three normal forms (as in the classical case)
$$ x_1^2 ,  \ x_1 x_2, \  x_1 x_2 + x_3^2.$$

 We also use the standard notation for partial derivatives, given a polynomial $G(x_1, \dots, x_n)$,
we denote
$$G_i : = \frac{\partial{G}}{\partial{x_i}}.$$
 
\section{Singular points of quartic surfaces in characteristic 2}

We consider a quartic surface $X = \{ F =0\} \subset \PP^3_K$, where $K$ is an algebraically closed field of characteristic equal to $2$,
and such that $X$ is normal, that is, $ Sing (X)$ is a finite set. 

If $X$ has a point of multiplicity $4$, then this is the only singular point, while 
if $X$ contains a triple point $P$, we can write the equation, assuming that the point $P$ is the point $x_1=x_2=x_3=0$:
$$ F (x_1, x_2,x_3, z) = z G (x) + B (x) ,$$
 where of course $G(x)$ is homogeneous of degree $3$ and $B(x)$ is homogeneous of degree $4$.

Recalling that  $G_i : = \frac{\partial{G}}{\partial{x_i}}, B_i : = \frac{\partial{B}}{\partial{x_i}}$, we have 
$$ Sing(X) = \{ G(x) = B(x) = G_i z + B_i = 0, \ i=1,2,3 \} \ .$$

If $(x,z) \in Sing(X)$ and $x \in  \{ G(x) = B(x) = 0\}$, then $ x \notin Sing (\{G=0\})$, since $ x \in Sing (\{G=0\}) \Rightarrow  x \in Sing (\{B=0\})$ and then the whole line $(\la_0 z, \la_1 x) \subset Sing(X)$. Hence $\nabla (G) (x) \neq 0$ and there exists a unique
singular point of $X$ in the above line. Since the two curves  $ \{ G(x) = 0\}, \{ B(x) = 0\}$
have the same tangent at  $x$ their intersection multiplicity at $x$ is at least $2$, and we conclude:

\begin{prop}\label{monoid}
Let $X$ be  quartic surface $X = \{ F =0\} \subset \PP^3_K$, where $K$ is an algebraically closed field,
and suppose  that $ Sing (X)$ is a finite set. If $X$ has a triple point then $ | Sing(X)| \leq 7$.

More generally, if  $X$ is a degree $d$ surface $X = \{ F =0\} \subset \PP^3_K$, where $K$ is an algebraically closed field,
and we suppose  that 
\begin{itemize}
\item
$ Sing (X)$ is a finite set, and 
\item
  $X$ has a  point of multiplicity $d-1$
  \end{itemize}
  
 then $$ | Sing(X)| \leq 1 + \frac{d (d-1)}{2}.$$

\end{prop}

\begin{proof}
The second assertion follows by observing that in proving the first we never used the degree $d$, except for
concluding that the total intersection number (with multiplicity) of $G,B$ equals $ d (d-1)$.

\end{proof}
 
Assume now that we have a double point $P$ of $X$ and we take coordinates such that $P = \{  x := (x_1, x_2,x_3) =0, z=1\}$, 
thus we can write the equation 
 $$ {\bf (Taylor \ development)}: \  \ F (x_1, x_2,x_3, z) = z^2 Q (x) +  z G (x) + B (x) ,$$
  where of course $Q,G,B$ are homogeneous of respective degrees $2,3,4$.

Then $$  \ Sing (X) = \{ (x,z) | G(x) =  z^2 Q (x)  + B(x) = z^2 Q_i (x) +  z G_i  (x) + B_i (x) =0, \ i=1,2,3\}.$$

We consider then the projection 
$$\pi_P : X \setminus \{P\} \ra \PP^2 : = \{ (x_1, x_2,x_3)\}.$$

\begin{lemma}\label{1-1}

 1) If $P$ is a singular point of the quartic $X$,  consider the projection $Sing (X) \setminus \{P\} \ra \PP^2$.
 
Two  singular points (different from $P$) can have the same image only if they map to the same point $x$ of 
 the finite subscheme $\Sigma \subset \PP^2$ defined by $Q=G=B=0$,  
 the three gradients $\nabla Q (x) ,  \nabla G(x) ,   \nabla B (x)$ are all proportional and   moreover $\nabla Q (x) \neq 0$,
$\nabla G (x) \neq 0$.

2) If $x \in \Sigma$, then there is a $z$ such that $(x,z) \in Sing(X)$ if 
$$ z^2 \nabla Q (x)  + z \nabla G(x)  +  \nabla B (x)= 0.$$
\end{lemma}

\begin{proof}
In fact, if a line $L$ through $P$ intersects $X$  in $2$ other singular points, then $ L \subset X$,
hence 
$$ L \subset \{ (x,z) | Q(x) = G(x) = B(x) = 0\} = : P * \Sigma,$$
where $\Sigma \subset \PP^2$ is the subscheme defined by $Q=G=B=0$,
and $\Sigma$  is a 0-dimensional subscheme since $X$ is irreducible.

If $x \in \Sigma$ and $(x,z), (x,z+ w) \in Sing(X)$ are different points, from the equations 
$$ z^2 \nabla Q (x)  + z \nabla G(x)  +  \nabla B (x)= (z^2 + w^2 )\nabla Q (x)  + (z + w ) \nabla G(x)  +  \nabla B (x)=0$$
follows $   \nabla G(x) = w \nabla Q (x) $, $ \nabla B (x) = z (z+w)  \nabla Q (x)  $.
In particular, it cannot be $\nabla Q (x) = 0$, and the three curves are all tangent at $x$.

Finally, if the three gradients are proportional, then we can find $(z,w)$ solving the equations
$   \nabla G(x) = w \nabla Q (x) $, $ \nabla B (x) = z (z+w)  \nabla Q (x)  $; and $w\neq 0$ if $   \nabla G(x) \neq 0$.

\end{proof}

\subsection{The 7 strange points of a plane quartic in characteristic $=2$}

In general, if $\{ g(x) = 0 \}$ is a plane curve of even degree $d=2k$ , from the Euler formula $\sum_i x_i g_i \equiv 0$
  we infer that if the critical scheme $\sC_g := \{\nabla g (x) =  0\}$  is finite,   then the first trivial
  estimate is  that its  cardinality is at most $(d-1)^2$, by the theorem of B\'ezout.
  
  Take in fact a line $L$ not intersecting this scheme $\sC_g$, and assume that $L = \{ x_3 = 0\}$: then 
  $\{ \nabla g (x) =  0\} = \{ g_1 = g_2=0, \ x_3 \neq 0 \}$, and this set has, with multiplicity, cardinality $(d-1)^2$ 
   or less if
   $g_1=  g_2 = x_3 = 0$ is non empty.
   
\bigskip

We consider   first the case of a general plane quartic curve $\{ B(x) = 0\}$.

The crucial observation is that  the space of (homogeneous) quartic polynomials $K [x_1,x_2,x_3]_4$
splits as  a direct sum 
$$K [x_1,x_2,x_3]_4 = \sQ \oplus  V,$$
where $\sQ$  consists of squares of (homogeneous)  quadratic polynomials, and V is spanned as follows: 

$$ V : = \langle x_i^3 x_j,  x_i x_1x_2x_3 \rangle.$$ 

$V$ has dimension $9$  and the group $GL(3,K)$ acts with finite stabilizer on the
Klein curve 

$$ B^0 : = x_1^3 x_2 + x_2^3 x_3 + x_3^3 x_1 $$

hence the orbit of $B^0$ is dense in $V$.

The gradient $\nabla B^0$ vanishes
at precisely $7$ points, the seven points $ \{ (1, \e, \e^5)| \e^7 =1\}$.

In fact,  $$ B^0_1 = x_1^2 x_2 + x_3^3 ,  \  B^0_2 = x_2^2 x_3 + x_1^3 , \ B^0_3 = x_3^2 x_1 + x_2^3 .$$

Using the first equation and taking the cubes of $x_2^2 x_3 = x_1^3$, we find (since $x_i \neq0$ for the points of $\sC$)
that $ y = \e x$, with $\e^7=1$, and $ z = \e^5 x$. Hence we get the seven points $ \{ (1, \e, \e^5)| \e^7 =1\}$.

We derive the following property.

\begin{prop}\label{7}
 For a homogeneous quartic polynomial $B \in K [x_1,x_2,x_3]_4$
 let $\sC_B$  be the critical locus of $B$
(where the gradient 
$\nabla B$ vanishes). 
If $\sC_B$  is a finite set, then it consists of at most  7 points.

For $B$ general, $\sC_B$  consists of exactly 7 reduced points.
\end{prop}

\begin{proof}

   Since $x_3 B_3 = x_1 B_1 + x_2 B_2$, we get that if $B_1=B_2 = 0$, then $x_3 B_3=0$.
   
   If $\{B_1=B_2 = 0\}$ is infinite, then on its divisorial part $C$ we have $x_3 =0$, since $\sC_B$ is finite.
   
   This means that $B = x_3 B' + \be(x_1,x_2)^2$, hence $\nabla B=0 \Leftrightarrow \nabla (x_3 B')=0$.
   
   If $x_3$ does not divide $B'$, then  $ | \sC_B|  \leq 6,$ since the gradient of $B'$ vanishes at most in $3$
   points. If $x_3$ divides $B'$, we get a contradiction to the finiteness of $  \sC_B$.
      
   We can therefore assume that we are in  the case where $\sB' : = \{B_1=B_2 = 0\}$ is finite. 
   
   Take a line containing $\de \geq 2$ points of $\sC_B \subset \sB'$, 
   say $x_3=0$, and let $x_1, x_2$ be coordinates 
   both not vanishing at these points.
   
   Then by Bezout's theorem $\sB'$ consists of at most $9$ points, counted with multiplicity.
   
   At any of the $\de \geq 2$ points of  $\sC_B \cap \{x_3=0\}$, the equation 
   $x_3 B_3 = x_1 B_1 + x_2 B_2$ shows that, since  $x_3 B_3$ vanishes of multiplicity at least $2$,
    a linear combination of
   $\nabla B_1, \nabla B_2$ vanishes, implying that  the local  intersection  multiplicity of $B_1, B_2$ is $\geq 2$.
   
   Hence $ | \sC_B| \leq 9 - \de \leq 7.$

For the second assertion, use the decomposition $K [x_1,x_2,x_3]_4 = \sQ \oplus  V,$
which allows us to write a general polynomial, after a change of variables, in the form
$$B = q^2 + B_0, \ q  \in K [x_1,x_2,x_3]_2.$$

We conclude since then $\nabla B = \nabla B_0.$

\end{proof}

\begin{rem}
(a) In Part II of this article we  show more generally  in a more elementary way  that if $B$ is a polynomial of even degree $=2m$,
and $\sC_B$ is finite, then its cardinality is at most $(d-1)(d-2) +1$, and that this estimate is sharp,
that is, equality holds in a Zariski open set of the spaces of such polynomials.

(b) A referee points out that a more general result (also in other characteristics) is  contained in Theorem 2.4 of \cite{liedtkecan},
whose statement  however does neither  mention derivatives nor critical sets, so that our clear statement
about $\sC_B$  is not easy to extract from the statement  of Theorem 2.4 of \cite{liedtkecan}.
\end{rem}

\subsection{The inseparable case}
We consider first the inseparable case where $G \equiv 0$, hence
$$ X  = \{   z^2 Q(x)  + B(x) = 0 \},  Sing (X) =  X \cap \{   z^2 \nabla Q (x) =  \nabla B (x) \}.$$

\begin{prop}\label{inseparable}
Let $X$ be  a normal quartic surface $X = \{ F =0\} \subset \PP^3_K$, where $K$ is an algebraically closed field
of characteristic  $2$,
hence $ Sing (X)$ is a finite set. If $X$ has a double point  $P$ as in (Taylor) 
such that the projection with centre $P$ is an inseparable double cover of $\PP^2$, i.e., $G \equiv 0$,
 then  $ | Sing(X)| \leq 16$, and there exists  a  case with $ | Sing(X)| =  14$.

\end{prop}

\begin{proof}
As in the proof of proposition \ref{monoid}, if a singular point $(x,z)$
satisfies $Q(x)=0$, then $B(x)=0$,
 and since $ Sing(X)$ is finite it must be   $\nabla (Q) (x) \neq 0$:  under this assumption
$(x,z)$ is the only  singular point lying above
the point $x \in \Sigma = \{ Q(x)= B(x)=0\}$.

Hence, in view of Lemma \ref{1-1} the projection $Sing(X) \setminus \{P\} \ra \PP^2$ is injective.

Conversely, if $x \in \Sigma  = \{ Q(x)= B(x)=0\}$,  it is not possible that $\nabla Q (x) =  \nabla B (x)=0$,
while for $\nabla Q (x) = 0,  \nabla B (x) \neq 0$ there is no singular point lying over $x$,
and for $\nabla Q (x)  \neq 0$ there is at most one singular point lying over $x$, and one iff 
$\nabla Q (x) ,  \nabla B (x) $ are proportional vectors.

Hence in the last  case the intersection multiplicity of $Q,B$ at $x$ is at least $2$, and
in particular  over $\Sigma$ lie at  most $4$ singular points.

\bigskip

We proceed now  in the proof considering several different cases, according to the normal form of $Q$.

\bigskip

{\bf Step I) 
We first consider the case where $Q$ is a double line, and show that in this case $X$ has at most  $8$ singular points}.

\medskip

We can in fact choose coordinates such that $ F = z^2 x_1^2 + B(x)$, hence the singular points $(x,z)$ are determined
by the equations $F=  \nabla B (x)=0$. And we have a bijection between $Sing(X)$ and the points of the plane with coordinates $x$ where
 $ \nabla B (x)=0$ and   $x_1 \neq 0$. In fact the points where  $ \nabla B (x)=x_1 = 0, B \neq 0$, do not come from singular points,
 while the points with $ \nabla B (x)=x_1 = B = 0$ would provide   infinitely many singular points.

    We are done  by Proposition \ref{7}  if $\sC_B$ is finite. In the contrary case, since $Sing(X)$ is finite, the only common divisor of the $B_i$'s
   is $x_1^a$.
   
    Since $x_1$ divides all the partial derivatives $B_i$, we can write $ B = q(x_2, x_3)^2 + x_1^2 B'$,
   and $\nabla B = x_1^2 \nabla B'$, and since $B'$ is not a square, we get at most 
   $2$ singular points.

   Hence Step I) is proven.

   \bigskip

{\bf Step II) 
We consider next the case where $Q$ consists of two lines, and show that in this case $X$ has at most $13$ singular points.}

\medskip

We can in fact choose coordinates such that $ F = z^2 x_1 x_2 + B(x)$, hence 
$$Sing(X) = \{z^2 x_1 x_2 + B(x)= 0 , z^2  x_2 + B_1(x)=  z^2 x_1  + B_2(x)=  B_3=0\}.$$
Hence the singular points satisfy 
$$ (**) \ B_3= B + x_1 B_1 = B + x_2 B_2 = 0,$$
where the last equation follows form the first two, in view of the Euler relation.

Let $(x)$ be  one of the solutions of $(**)$: we find at most one  singular point lying above it
if $x_1 \neq 0$, or $x_2 \neq 0$.
If instead $x_1 = x_2 =0 $ for this point, then $B =B_3 =0$, and either there is no singular point of $X$ lying over it,
or $(x)$ is a singular point of $B$, and we get infinitely many singular points for $X$, a contradiction.
Hence  it suffices to bound the cardinality of this set.

If the solutions of $(**)$ are a finite set, then their cardinality is $\leq 12$, and also $|Sing(X)| \leq 13$.

If  instead $(**)$ contains an irreducible curve $C$,  this factor $C$ cannot be   $x_1=0$ or $x_2=0$, since 
 for instance in the first case then  $ x_1 | B \Rightarrow x_1 |F$, a contradiction.
 
 We find then infinitely many points satisfying $(**)$ and with $x_1x_2 \neq 0$. Over these lies a singular 
 point if $ z^2 = B_1/ x_2 = B_2 /x_1 =  B / (x_1 x_2)$ is satisfied. 
 
 But if the first equation is verified, then also the others follow from $x_1 B_1 + x_2 B_2 = x_3 B_3 = 0$,
 respectively $ B + x_1 B_1 = 0$.

Hence the existence of such a curve $C$ leads to the existence of infinitely many singular points of $X$
 and Step II) is proven by contradiction.

\bigskip

{\bf Step III)
We consider next the case where $Q$ is smooth, and show that in this case $X$ has at most $16$ singular points.}

\medskip

We can in fact choose coordinates such that $Q(x) =  x_1 x_2 +  x_3^2$,
hence $ F = z^2  Q(x) + B(x) = z^2 x_1 x_2 + z^2 x_3^2 +  B(x)$.

Here
$$Sing(X) = \{z^2 x_1 x_2 +  z^2 x_3^2 + B(x)= 0 , z^2  x_2 + B_1(x)=  z^2 x_1  + B_2(x)=  B_3=0\}.$$
Hence the singular points satisfy 
$$ (*) \ B_3= B + x_1 B_1 + z^2 x_3^2 = B + x_2 B_2 + z^2 x_3^2 = 0,$$
and again here  the last equation follows form the first two, in view of the Euler relation.

If $x_1 \neq 0$, or $x_2 \neq 0$, there is at most one singular point lying over (we mean always:  a point different from $P$) 
the point $x$. The same for $x_1= x_2=0$
since then $x_3^2 \neq 0$.

Multiplying the second equation of (*) with $x_2$, and the last with $x_1$, we get, for the singular points,
$$ B x_2  + Q B_1  =B x_2  + x_1 x_2 B_1 + B_1 x_3^2 = 0$$
$$ B x_1  + Q  B_2 = B x_1  + x_1 x_2 B_2 + B_2 x_3^2 =  0.$$

Hence the singular points different from $P$ project injectively into the set
$$ \sB : = \{x|  B_3 = B x_2  + Q B_1  =B x_1  + Q  B_2= 0 \} \supset  \{x|  B_3 = B   = Q  = 0 \}.$$

By B\'ezout $\sB$ consists of  at most $15$ points, unless  $B_3 , B x_2  + Q B_1,  B x_1  + Q  B_2$ have a
 common component.

If $\sB$  contains an irreducible  curve $C$, since either $x_1 \neq 0$ or $x_2 \neq 0$ for the general point of $C$,
 there exists $ i \in \{1,2\}$ such that the cone $\Ga$ over $C$ has as open nonempty subset the cone $\Ga'$ over $C' : = C \setminus \{x_i=0\}$ which is contained in   the set of  solutions of $(*)$. Hence $\Ga$ is contained in   the set of  solutions of $(*)$.

Argueing similarly,  if  $C \neq \{x_1=0\}, C \neq \{x_2=0\}$, the cone $\Ga$ over $C$ has  as nonempty open subset 
the cone $\Ga''$ over $C'' : = C \setminus \{x_1 x_2 =0\}$ which is contained in the solution set 
of $z^2  x_2 + B_1(x)=  z^2 x_1  + B_2(x)=  B_3=0$, hence $\Ga$ is contained in this solution set. 
 We conclude that  $Sing(X)$ contains $ X \cap \Ga$,
hence it is an infinite set, a contradiction.

By symmetry, it suffices to exclude the possibility that $\sB$ contains $C = \{x_1 = 0\}$.
In this case we would have that  $x_1 | B_3, B_2, B x_2 + Q B_1$. Since $x_1 $ does not divide $Q$, it follows  that the
plane $\Ga  = \{x_1 = 0\}$ has  as nonempty open subset the cone over $C \setminus \{Q=0\}$
which is contained in  the set $\{z^2  x_2 + B_1(x)=  z^2 x_1  + B_2(x)=  B_3=0\}$,  hence
$\Ga$ is contained in this set and  $Sing(X)$ contains $ X \cap \Ga$, again
 a contradiction.

Hence Step III) is proven.

\begin{rem}
In part II, using the Hilbert-Burch theorem, we  show that, in the case of Step III), the
number of singular points is at most $14$.
\end{rem}

\bigskip

{\bf Step IV) 
We  construct now a case where there are other $13$ singular points beyond $P$, hence $X$ has $14$ singular points.}

By the previous steps, we may assume that $Q(x) =  x_1 x_2 +  x_3^2$, and we take $B = y_1 y_2 y_3 y_4$,
where  $y_1,  y_2,  y_3$ are independent linear forms and  $y_4 =  y_1 + y_2 + y_3$.

Recall that 
$$ Sing (X)  = \{   z^2 Q(x)  + B(x) = 0 ,   z^2 \nabla Q (x) =  \nabla B (x) \}.$$
Multiplying the second (vector) equation by $Q$ we get 
the equation 
  $$ (*) \  B(x)   \nabla Q (x) =  Q(x)  \nabla B (x) \Leftrightarrow   \nabla (QB) (x) =0.$$

  The solutions of (*) consist of 
  \begin{enumerate}
  \item
  the points where $Q(x) = 0 $, hence $Q=B=0$: these are precisely $8$ points, for general choice of the linear forms $y_i$;
  and they are not projections of singular points of $X \setminus{P}$, as the gradients $ \nabla B,  \nabla Q$ are linearly independent
  at them;
    \item
  the points where $Q(x) \neq 0 $, $B(x)=0$,  hence $ B = \nabla B =0$; 
   that is, the $6$ singular points $y_i = y_j=0$ of $\{ B(x)=0\}$, 
  giving rise to the $6$ singular  points of $X$ with $z=0$,
  \item
  (possibly) a point where $\nabla Q (x) =    \nabla B (x)=0$ but $Q \neq 0, B\neq0$, this comes from exactly one singular point of $X$;
    \item
  points satisfying $$(**) \ Q(x) \neq 0 \neq  B(x), \nabla (Q) \neq 0 \neq \nabla (B), B(x)   \nabla Q (x) =  Q(x)  \nabla B (x).$$.
  \end{enumerate}
  
Observe that  $\nabla Q (x) = (x_2,x_1,0) $ vanishes exactly at the point $x_1=x_2=0$, while 
in the coordinates $(y_1, y_2, y_3)$ we have
$$\ ^t \ \nabla (B) (y) = (y_2  y_3 (y_2 + y_3), y_1  y_3(y_1 + y_3),y_1 y_2 (y_1 + y_2)), $$
hence the gradient $\nabla (B) (y)$ vanishes exactly at the $6$ singular points of $B$, and at the point $y_1+ y_2=y_1 +y_3=0$.

We have that this point is the point $x_1=x_2=0$ as soon as $y_1 = y_3 + x_1, y_2 = y_3 + x_2$ (then $y_4 =  y_3 + x_1 + x_2$).

Going back to the notation of Step III), we consider then the set  
$$ \sB : = \{x|  B_3 = B x_2  + Q B_1  =B x_1  + Q  B_2= 0 \} \supset  \{x|  B_3 = B   = Q  = 0 \}.$$

Since $ B = (y_3 + x_1)(y_3 + x_2) y_3 (y_3 + x_1+ x_2)$, if we set $ y_3 = a_1 x_1 + a_2 x_2 + a_3 x_3$, with $a_3 \neq 0$, 
then using
$$B_i = \partial (y_1 y_2 y_3 y_4) / \partial x_i = \sum_1^4  \frac{B}{y_j} (y_j)_i, \ $$
 $$ B_3 = 0 \Leftrightarrow   \sum_1^4  \frac{B}{y_j} = 0, \ B_1 = a_1 B_3 + y_2 y_3 y_4 + y_1 y_2 y_3, \ B_2 = a_2 B_3 + y_1 y_3 y_4 + y_1 y_2 y_3,$$
 the equations of $ \sB$ simplify to
 $$ B_3 = 0 , \ y_2 y_3 ( y_1 y_4 x_2 + Q (y_1 + y_4)) = 0 , y_1 y_3 (y_2 y_4 x_1 + Q (y_2 + y_4)) = 0.$$ 
 
 Since we are left with finding solutions where $B\neq0$, the equations reduce to
 $$ B_3 = 0 , \  [(y_3 + x_1) (y_3 + x_1+ x_2)   + Q ] x_2  = 0 ,  [(y_3 + x_2) (y_3 + x_1+ x_2)  + Q ] x_1 = 0.$$
 We already counted the point $x_1 = x_2 =0$. If $x_1 =0$ and $x_2 \neq 0$, 
 we find $$ B_3 = y_3^2 + y_3 x_2 + x_3^2 = 0 \ ,$$
 but since we observe that $B_3 \equiv 0$ on the line $x_1=0$, we get the two points 
 $$ x_1 = y_3^2 + y_3 x_2 + x_3^2 = 0 \ \Leftrightarrow x_1=0, (a_3^2 +1) x_3^2 + (a_2^2 + a_2) x_2^2 + a_3 x_2 x_3=0.$$
 Similarly 
  if  $x_2 =0$ and $x_1 \neq 0$
 we get  the two points $$ x_2 = y_3^2 + y_3 x_1 + x_3^2 = 0 \  \Leftrightarrow x_2=0, (a_3^2 +1) x_3^2 + (a_1^2 + a_1) x_1^2 + a_3 x_1 x_3=0.$$
 If both $x_1 \neq 0, x_2 \neq 0$, 
 we find $$ B_3 =x_1^2 + x_2^2 + y_3 (x_1 + x_2) = (y_3 + x_1)^2 + y_3   x_2   + x_3^2= 0 \ .$$
 The second equation (of the three above)  is reducible, it equals $$(x_1 + x_2 + y_3) (x_1 + x_2)=0.$$
  Again $B_3 \equiv 0$ on the line $x_1=x_2$,
 while the points with $(x_1 + x_2 + y_3)=0$ yield the line $y_4=0$ which is contained in $B$, hence we do not need to consider these points.
 
 Hence we get  two more solutions:
 $$ x_1 = x_2 , y_3^2 + x_1^2 + y_3   x_1   + x_3^2= 0,\Leftrightarrow$$
 $$ \Leftrightarrow x = (x_1, x_1, x_3) , (a_3^2 +1) x_3^2 + (a_2^2 + a_1^2 + a_2 + a_1+ 1) x_1^2 + a_3 x_1 x_3=0.$$
 and $X$ has exactly $1 + 6 + 1 + 2+2+2= 14$ singular points,
 $$ (1): x=0, z=1, $$
 $$ (6):  z=0, x = (0,1,a_2), (0,1, 1 + a_2), (1,0, a_1), (1,0, 1 + a_1), (1,1, a_1 + a_2), (1,1,  1 + a_1 + a_2), $$
 $$ (1): ( 1,0,0,1),$$
 $$ (2):  x= (0, 1, b) , \ (a_3^2 +1) b^2 + (a_2^2 + a_2)  + a_3 b =0$$
 $$ (2):  x= (1,0, c) , \ (a_3^2 +1) c ^2 + (a_1^2 + a_1)  + a_3 c =0.$$
 $$ (2):  x= (1,1, d) , \  (a_3^2 +1) d^2 + (a_2^2 + a_1^2 + a_2 + a_1+ 1)  + a_3 d =0.$$
 One can now verify that, for general choice of the $a_i$'s (which can be made explicit requiring  $a_3 \neq 1$,
 $ b \neq a_2, 1 + a_2$, $ c \neq a_1, 1 + a_1$, $ d \neq a_1 + a_2, 1 + a_1 + a_2$), we obtain $14$ distinct points.
 
\end{proof}

 An interesting question posed by Matthias Sch\"utt is: how many of the 14  singular points may be defined over $\FF_2$, $\FF_4$?
 
 Over $\FF_2$, there are alltogether $15$ points in $\PP^3_{\FF_2}$, and each coordinate plane contains $7$ points:
 but a plane section of $X$ cannot have $7$ singular points, hence there cannot be $\geq 12$ singular
 points defined over  $\FF_2$.
 
 For  $\FF_4 =  \FF_2[u]/ (u^2 + u + 1)$, we try with
 $ a_3=u, a_1 =u^2, a_2=u^2$.
 
 The above conditions mean that  $c,d \neq u, u^2$, $ d \neq 0,1$ and we see that 
 $u,u^2$ are not roots of $b^2 + u^2 +b=0$, while $0,1$ are not roots of
 $u d^2 + 1 + ud=0$, which when multiplied by $u^2$ becomes $ d^2 + u^2 + d=0$.
 
 We conclude that
 
 \begin{prop}
 There exists a normal quartic surface $X$ with $14$ singular points defined over $\FF_{16}$,
 $8$ of them defined over $\FF_4$, $4$ of them defined over $\FF_2$.
 
 \end{prop}
 \begin{proof}
 Choosing  $ a_3=u, a_1 =u^2, a_2=u^2$ we get, on top of the two points $ x=0, z=1, \  ( 1,0,0,1) $,  the 12 points
 with $z=0$ and with 
  $$  x = (0,1,u^2), (0,1, u), (1,0, u^2), (1,0, u), (1,1, 0), (1,1,  1 ),  (0, 1, b) ,  (1,0, c) ,  (1,1, d) , $$
 where $b,c,d $ are roots of the quadratic equation $ z^2 + z + u^2=0$.
 
 This equation has no root in $\FF_4$, hence it defines a quadratic extension of $\FF_4$.
 
 \end{proof}

 \bigskip
 
 \subsection{More on the inseparable case}
Here is another construction of  quartics with $14$ singular points,  due to Matthias Sch\"utt.

Consider the quartic $X$ of equation

$$X : = \{  F (w,x_1,x_2,x_3) : = w^4 + w^2 x_1^2 + B(x) = 0\}.$$

 Here $B(x) : = B(x_1,x_2,x_3)$ is homogeneous of degree $4$ and
the singular points are the solutions of 
$$ \nabla B(x)= 0 , w^4 + w^2 x_1^2 + B(x) = 0.$$

For each $x$, the polynomial $w^4 + w^2 x_1^2 + B(x)$ is the square of a quadratic polynomial in $w$,
which is separable if $x_1\neq 0$.

Let $$\sC : = \{ x \in \PP^2 | \nabla B(x)= 0 \}.$$

Hence $ |Sing(X)| = 2 | \sC|$ provided $ \sC \cap \{x_1=0\} = \emptyset,$
a condition which can be realized for the choice of a general  linear form
once we find a quartic $B$ with $\sC$ finite.

 By Proposition \ref{7}, follows that,  for a general quartic polynomial $B$, the locus  $\sC$ consists of
$7$ reduced points, hence we get $X$ with $14$ singular points, which are nodes.

 If we want to be more explicit, the first choice is to take $B$ the product of $4$ general linear forms, as in Step IV)
of proposition \ref{inseparable}.

The second explicit choice, as already mentioned, is to take the Klein quartic
$$ B : = x_1^3 x_2 + x_2^3 x_3 + x_3^3 x_1.$$

Recall that the set $\sC$ has always cardinality at most $7$ in view of Proposition  \ref{7},
so this construction leads to no more than $14$ singular points, and in general to $14$
singular points.

\subsection{The case where a  variable appears only with even multiplicity}

In Part II of this article, using elaborate arguments, we shall show that if $X$ is defined by
$$ X = \{ (x_1,x_2,x_3,z) | \la z^4 + z^2 Q(x) + B(x) =0\},$$
and $X$ is normal, then $X$ has at most $14$ singular points, and in general it has $14$ 
singular points which are nodes.

It is still an open question whether all the quartics with $14$ nodes belong to this family.

We prove here a weaker result with an elementary proof.

\begin{theo}\label{z-square}
Assume that the normal quartic surface $X$
  is defined by an equation of the form 
$$ X = \{ (x_1,x_2,x_3,z) | F (x,z) := \la z^4 + z^2 Q(x) + B(x) =0\}.$$

Then $ | Sing(X)| \leq 16$.

\end{theo}

\begin{proof}
1) The case $\la=0$ was dealt with in Proposition \ref{inseparable}.

2) The case where $Q \equiv 0$, $\la=1$ gives rise, again by Proposition \ref{7}, to at most 7 singular points, which
are in general 7 $A_3$-singularities.

 3) The case
where $Q(x)$ is the square of a linear form was dealt with in the previous subsection.

4) There remains to treat  the case  where $\la=1$ and $Q$ is a smooth conic, $Q(x) = x_1 x_2 + \mu x_3^2$
in suitable coordinates.

Then 
$$\nabla F=0 \Leftrightarrow z^2 x_2 + B_1 = z^2 x_1 + B_2 = B_3 =0.$$

If a singular point has $x_2 \neq0, $ or $x_1 \neq0$, then $z$ is uniquely determined
by its projection in the plane with coordinates $(x)$, which must lie 
in the set
$$ \sF_2 : = \{B_3= B_1^2 + Q B_1x_2 + x_2^2 B=0\},$$
and also in the set
$$ \sF_1 : = \{ B_3= B_2^2 + Q B_2 x_1  + x_1^2 B=0\}.$$
Actually, $\sF_2 \cap  \{x_2 \neq 0\}$ is in bijection with $Sing(X) \cap \{ x_2 \neq0\}$,
and similarly $\sF_2 \cap  \{x_1 \neq 0\}$ is in bijection with $Sing(X) \cap \{ x_1 \neq0\}$,
hence $ \sF_1,  \sF_2$ are finite unless $x_1$ divides $B_2, B_3$, respectively 
$x_2$ divides $B_1, B_3$.

Assume that $\sF_2$ is finite: then by the theorem of B\'ezout it consists of $18$ 
points counted with multiplicity. Since $x_2$ does not divide $B_1, B_3$,
and we notice that 
$$ B_1 \equiv c x_1^2 x_3 + d x_3^3 \ (mod \ x_2) , B_3 \equiv c x_1^3 + d x_1 x_3^2 \ (mod \ x_2),$$
the length two subscheme $ \sA : = \{ x_2= c x_1^2 + d  x_3^2=0\}$ is contained in $\sF_2$.

Since  $B_1^2 + Q B_1x_2 + x_2^2 B$ lies in the square of the ideal generated by $x_2, B_1$,
at the point of $\sA$ the intersection multiplicity is at least $4$, hence $\sF_2$
consists of at most $15$ points.

Similarly, if $\sF_1$ is finite, it  consists of at most $15$ points.

Finally, if $x_2$ divides $B_1, B_3$, and $x_1$ divides $B_2, B_3$, then we can write
$$ B (x) = q(x)^2 + c x_1 x_2^3 + d x_1^3 x_2.$$

Hence $B_3$ is identically zero, and, setting $\phi(x) : = c  x_2^2 + d x_1^2 $,
we get 
$$ B_1 = x_2 \phi, \ B_2 = x_1 \phi .$$

Then $$ Sing(X) = \{ F=0, x_2 (z^2 + \phi(x)=0, x_1 (z^2 + \phi(x)=0\} \supset \{ F= z^2 + \phi(x)=0\},$$
and $X$ is not normal.

We conclude since over the point $\{x_1=x_2=0\}$ lie at most $2$ singular points,
hence $|Sing(X)| \leq 16$.

\end{proof}

\begin{rem}
Case 2) above (equation $z^4 + B(x)=0$) ,  which
 in general gives rise to 7 $A_3$-singularities
 is quite interesting.
 
  Because the minimal resolution $S$ of $X$ has then Picard number 22,
hence one sees immediately here that$S$ is a supersingular K3 surface.
\end{rem}

 \section{Inequalities provided by the Gauss map}
 
 The Gauss map $\ga : X \dasharrow \sP : = (\PP^3)^{\vee}$ is the rational map given by 
$$ \ga(x) : = \nabla F (x), \ x \in X^* : = X \setminus Sing(X).$$

We let $Y : = \ga(X)$ be the image of the Gauss map,  which is a morphism on $X^*$, and becomes a morphism $\tilde{\ga}$ 
on a suitable blow up $\tilde{S}$ of the minimal resolution $S$ of $X$.  The image $Y = \ga(X)$ is called the dual variety of $X$.

In order to compute the degree of $Y$ (this is defined to be  equal to zero if $Y$ is a curve), we consider a line $\Lambda \subset \sP$
such that  $\Lambda$ is  transversal to $Y$, this means: 

1) $\Lambda \cap Y = \emptyset $ if $Y$ is a curve

2) $\Lambda $ is not tangent to $Y$ at any smooth point, and neither contains   any singular point of $Y$,
nor any point $y$ where the dimension of the fibre $\tilde{\ga}^{-1} (y)$ is $=1$, 
so that 

3) $\Lambda \cap Y  $ is in particular a subscheme consisting  of $deg(Y)$ distinct points,
and its inverse image in $\tilde{S}$ is a finite set.

By a suitable choice of the coordinates, we may assume that 
$$\ga^{-1} ( \Lambda ) \subset   X \cap \{F_1 = F_2 = 0\}.$$

The latter  is a finite set, hence by Bezout's theorem it consists of $4 \cdot 3^2= 36$  points counted with multiplicity,
including the singular points of $X$. 

We get then the well known  formula

$$ (DEGREE -  FORMULA) \ \  deg(\ga) deg(\ga(X) ) = 36 - \sum_{P \in Sing(X)} (F,F_1, F_2)_P,$$
where $(F,F_1, F_2)_P$ is the local intersection multiplicity at $P$,
equal to 
$$ dim_K ( \hol_{X,P} / ( F_1, F_2)) = dim_K ( \hol_{\PP^3,P} / (F,  F_1, F_2)) .$$ 

Since $P$ is a singular point (actually we are interested in the case where   $F$ vanishes of order exactly $2$), we have
$$ (F,F_1, F_2)_P \geq 2  \ \forall P \in Sing(X) .$$ 

\begin{rem}
 If $P$ is a uniplanar double point, then the Taylor development of $F$ has the form 
$ F = \ell(x)^2 + g(x)$, where $\ell(x)$ is linear and $g$ vanishes of order at least $3$. Hence $F_1, F_2$ vanish of order at least $2$,
so that $ (F,F_1, F_2)_P \geq 8$.

If we have a biplanar double point, $ F = \ell_1(x) \ell_2(x)  + g(x)$, where $g$ vanishes of order at least $3$,
then for general choice of coordinates $x_1, x_2$, $F_1, F_2$ vanish of order $1$ and $ \ell_1(x), \ell_2(x)$
are in the ideal generated by $F_1, F_2$, hence by semicontinuity we have always $ (F,F_1, F_2)_P \geq 3$.

Indeed, a biplanar double point is (\cite{quartics}) an $A_n$ singularity, and for this singularity the contribution
is at least $n+1$.
\end{rem}

We observe moreover:

\begin{lemma}\label{line}
(i) Assume that $P \in Sing(X)$ is a node and $E$ 
the exceptional curve in the minimal resolution  $S$ of $X$.
Then $E$ maps, via the Gauss map,  to a line via an inseparable map of degree
 two. In particular the Gauss map cannot be birational if $X^{\vee} $ is a normal surface.

(ii) The Gauss image of $X$ cannot be a line.
\end{lemma}\label{node-line}
\begin{proof}
 (i): given a node $P$, an $A_1$-singularity, then
the affine Taylor development at  $P$ is given by
$$F = xy + z^2 + \psi (x,y,z) =0$$
and the Gauss map on the exceptional conic $ E \subset \PP^2, E = \{ xy + z^2=0\}$
is  given by $(x,y,0,0)$.

 If  $X^{\vee} $ is a normal surface, then 
$$\tilde{\ga} : \tilde{S} \ra X^{\vee} $$ is an isomorphism over the complement of a finite number of points of
$X^{\vee} $, a contradiction since $E$ maps $2$ to $1$ to a line.

(ii):  if $X^{\vee}$ is a line, then there are projective coordinates in $\PP^3$ such that
$$ X = \{ a z^4 + b w^4 + c z^2 w^2 + z^2 D(x,y) + w^2  E(x,y) + f (x,y) = 0 \}.$$

Writing $$ D(x,y) = d_1 x^2 + d_2 y^2 + d xy,  E(x,y) = e_1 x^2 + e_2 y^2 + e xy,$$
$$ f(x,y) = q(x,y)^2 + f_1 x^3 y + f_2 x y^3,$$
we see that 
$$ Sing(X) = X \cap \{  y M = x M = 0\}, \ M = d z^2 + e w^2 +  f_1 x^2 + f_2  y^2,$$
hence $ Sing(X) \supset X \cap \{ M=0\}$ and $X$ is not normal.

\end{proof}

From the above considerations follows: 

\begin{prop}\label{gaussestimate}
If $X$ is a normal quartic surface in $\PP^3$, with singular points of multiplicity $2$, then  $ \nu : = | Sing(X)| \leq 16$.

Equality holds only if  all the singularities are nodes, except possibly a singularity of type $A_2$ in the case where
the Gauss image $Y$ of $X$ is a plane.

If $X$ has a uniplanar double point, then $ \nu : = | Sing(X)| \leq 15$.
\end{prop}
\begin{proof}

 First of all, the inequality $\nu \leq 16$  is proven in Theorem \ref{z-square} if some variable, say $z$, occurs only
with even exponent. This condition is equivalent to $F_z \equiv 0$, and to the fact that the image is contained in
a plane.

If the number of singular points is $\geq 13$, necessarily by the degree formula follows
that we must have some node, since $13 \times 3 = 39 > 36$.

We can then apply  Lemma \ref{line}  which says that the image of the Gauss map contains a line and cannot consist only of one line:
hence the image will be an irreducible surface $Y:= \ga(X)$ of   degree  $\geq 2$.

From this follows then that either $ deg(\ga) \geq 2$, or $deg(\ga) =1$ and $deg(Y) \geq 3$,
since in this case $X$ is the dual of $Y$ by the biduality theorem, hence $Y$ cannot be a quadric.

The inequality $\nu \leq 16$ follows then from the degree formula,  and  in case of equality (since
  the only singular points which give a contribution 
  
  $(F,F_1, F_2)_P=2$ are the nodes)
 we have then at least $15$ nodes and a singularity of type $A_1$ (a node) or $A_2$.
  
Hence the minimal resolution of $X$ is a K3 surface, therefore it cannot be birational to a cubic surface:
hence we conclude that  $ deg(\ga) deg(Y) \geq 4$ and we must have $16$ nodes.

\end{proof}

In this Part I we show with elementary arguments that $|Sing(X) | \leq 16$, but in Part II, in collaboration with 
Matthias Sch\"utt,  we shall use
the fine theory of elliptic fibrations in characteristic 2 and their wild ramification in order to
prove the optimal bound $|Sing(X) | \leq 14$.

We briefly give now the flavour of the geometric arguments which shall be used in Part II.

Consider two singular points of $X$, say $P_1, P_2$, and the pencil of planes containing the line 
$L : = \overline{P_1 P_2}$,
corresponding to the linear system $ |H - P_1-P_2|$. 

Projection with centre $L$ provides a rational map 
$$ \pi_L : X \dasharrow \PP^1,$$
which is a fibration with fibres of arithmetic genus at most $1$.

The  basic observation is that, if  the general plane sections with planes in $ |H - P_1-P_2|$ have a singular point 
which is not a point in $Sing(X)$  (this means that we have a so-called quasi-elliptic fibration) then the dual line $L^{\vee}$ is contained in the dual variety $Y = \ga(X)$.

\begin{ex}

In the case where we have a double point $P$ of   inseparable type,
we have the equation
$$ \{ F (z,x) = z^2 Q(x) + B(x) =0 \}.$$
Then $\frac{\partial F}{ \partial z }=0$ at all points of $X$, and the image $Y = \ga(X)$ is contained in
a plane. Moreover the Gauss map $\ga$ factors through the inseparable double cover $ \pi_P : X \ra \PP^2$
such that  $ \pi_P (z,x) = x$, since 
$$ \nabla F (z,x) = z^2 \nabla Q(x)  + \nabla B(x)  = (B  \nabla Q + Q \nabla B)(x) .$$

If furthermore $Q$ is the square of a linear form, then
$$\ga(z,x) = (0, B_1(x), B_2(x), B_3(x)).$$
The base scheme $\{x | B_i(x)=0, i=1,2,3\}$ is finite (since $X$ is normal) and has length at most $7$ by lemma \ref{7}, hence 
we get a map  $\PP^2 \dasharrow \PP^2$ of degree $\geq 2$. 

The conclusion is that  we have at most $8$ singular points and  that $deg (\ga) \geq 4$.

 The  estimate that we get, $|Sing(X) | \leq 8$,  works perfectly since the point $P$ is uniplanar
 and gives a contribution $ \geq 18$.

\end{ex}  
\bigskip

We reach then our final result:

\begin{theo}\label{mt}
Let $X$ be  quartic surface $X = \{ F =0\} \subset \PP^3_K$, where $K$ is an algebraically closed field
of characteristic  $2$,
and suppose  that $ Sing (X)$ is a finite set. 

Then  $ | Sing(X)| \leq 16$.

In fact:

(1) if $ | Sing(X)| =  16 $, then the  only singularities of $X$ are 
 nodes,  except possibly a singularity of type $A_2$ in the case where
the Gauss image $Y$ of $X$ is a plane.  The minimal resolution
$S$ of $X$ is a minimal K3  surface with Picard number at least   $ | Sing(X)| + 1 =  17$.

(2) If $X$ contains a uniplanar double point, then $ | Sing(X)| \leq  14$, and  if equality holds,
all other singularities are nodes, and  $Y$ is a plane.

\end{theo}

\begin{proof}
If $X$ has a point of multiplicity $4$, then $ | Sing(X)| =1$.

If $X$ has a point of multiplicity $3$, then by Proposition \ref{monoid} $ | Sing(X)| \leq 7$.

If $X$ has a point of multiplicity $2$ such that the projection with centre $P$ is inseparable, then $ | Sing(X)| \leq 16$ by Proposition \ref{inseparable}.

 (ii) of Lemma \ref{line} shows that the image $X^{\vee}$ of the Gauss map is a surface,
and Theorem \ref{z-square} shows that if $X^{\vee}$ is a plane, then  $ | Sing(X)| \leq 16$.

(i) of Lemma \ref{line} shows that the image $X^{\vee}$ is not a normal surface if $deg (\ga)=1$,
hence $deg(X^{\vee}) deg(\ga) \geq 3$ and indeed equality holds only if $deg(\ga)=1$.

The inequality $ | Sing(X)| \leq 16$ and items (1) and (2) follow from proposition \ref{gaussestimate}.

\end{proof}

 \begin{rem}\label{K3}
 Let $f : S \ra X$ be the minimal resolution of a quartic surface with only double points $p_1, \dots, p_k$ as singularities.
 
 Then the inverse image of each $p_i$ is a union of irreducible curves $E_{i,1}, ,\dots, E_{i, r_i}$
 and if $ r : = \sum_1^k r_j \geq k$, we have irreducible curves $E_1, \dots, E_r$ such that the intersection matrix 
 $\langle E_i , E_j\rangle$ is negative definite \cite{mumford}. 
 
 Hence it follows that the Picard number $\rho(S) \geq r+1 \geq k+1$, keeping in consideration that the hyperplane section $H$ is orthogonal to the $E_i$'s. More generally, each singular point $P$ contributes $r(P)$ to the Picard number $\rho(S)$ if the local resolution
has $r(P)$ exceptional  curves.  Then  
$$\sum_P r(P) + 1 \leq \rho(S) \leq b_2(S),$$ where the  lower bound for the second Betti number is obtained 
 using $\ell$-adic cohomology, see for instance
 \cite{milne} cor. 3.28 page 216.

 If $S$ is a minimal K3 surface, we  have $\chi(S) = 2, b_2(S) = 22$.
  
 Now, it follows that $S$ is a minimal K3 surface if the singular points are rational singularities (this means that  $\sR^1f_* (\hol_S)=0$): 
  because these are then rational double points (see \cite{artin}) and $K_S$ is a trivial divisor. However this does not need to hold
  in general. 
 
 If instead the singular points are not rational, it follows that $K_S$ is the opposite of an effective 
 exceptional divisor and  $h^1(\hol_S) >0$.  $ K_S^2$ is    negative, and $\chi(S)$ nonpositive; if both are negative ($h^1(\hol_S) \geq 2$) by Castelnuovo's theorem $S$ is ruled and
 possibly non minimal. Hence in this case we do not have an explicit upper bound for $b_2(S) = - K_S^2 + 12 \chi(S)$.
 \end{rem}

\bigskip

\begin{rem}

The K3  surfaces with  $\rho(S) = b_2(S) = 22$ are the so-called Shioda-supersingular K3 surfaces.

Shioda observes  \cite{shioda} that Kummer surfaces in characteristic $2$ have 
at most $4$ singular points, and proves that the Kummer surface associated to a product of elliptic
curves has $\rho(S) \leq 20$.

Rudakov and Shafarevich \cite{rudakov-shafarevich} described the supersingular K3 surfaces in characteristic $2$ according to their Artin invariant $\s$, which determines the intersection form on $Pic(S)$.

 Shimada \cite{Shimada} and then Dolgachev and Kondo \cite{dolga-kondo} constructed a supersingular K3 surface in characteristic $2$ with Artin invariant $1$
and with $21$ disjoint $(-2)$
curves, but  this surface does not  have  an embedding as a quartic surface with $21$ ordinary double points
(see also \cite{katsurakondo}), as we saw in the proof of theorem \ref{mt}.

In their case the orthogonal to the $21$ disjoint $(-2)$
curves is a divisor $H$ with $H^2=2$, yielding a realization as a double plane.

 A  supersingular K3 surface can be birational to a nodal   quartic surface: indeed, the examples given here with $14$ nodes
are supersingular, being unirational, see \cite{shiodass}.

Artin proved \cite{artinSS} that K3  surfaces with height of the formal Brauer group $h$ ($h \in \NN \setminus\{0\}  \cup \{ \infty \}$)
   satisfy $ \rho(S) \leq 22 - 2h$ if the formal Brauer group is $p$-divisible. 
   
   Artin observes that K3 surfaces $S$ with $ \rho(S) = 21$ do not exist, as he proves that 
   if $h = \infty$ and $S$ is elliptic, then the formal Brauer group is $p$-divisible,
   and moreover $S$ is elliptic once $\rho(S) \geq 5$ by Meyer's theorem (see \cite{serre}
   corollary 2, page 77).
   
   Artin predicted (modulo the conjecture that $h = \infty$ implies that $S$ is elliptic)
 that  $h = \infty \Leftrightarrow  \rho(S) = 22$.
 
 This equivalence follows from the Tate conjecture for K3 surfaces over finite fields, as explained in \cite{liedtke},
 discussion in section 4, especially theorem 4.8 and remark 4.9; the Tate conjecture was proven in $char=2$
 by Charles, Theorem 1.4 of \cite{charles},  and by Kim-Madapusi Pera \cite{kimmp}.

\end{rem}

\section{Symmetric Quartics}

One can try to see whether, as it happens for cubic surfaces or for quartics in characteristic $\neq 2$, one can construct quartics
with the maximum number of singular points  as  quartic surfaces admitting $\mathfrak S_4$-symmetry.

By the theorem of symmetric functions, every such quartic $X$ has an equation of the form

$$ F (x) : = F(a,\be, x) : =  a_1 \s_1^4 + a_2 \s_1^2 \s_2 +  a_3 \s_1 \s_3 + a_4  \s_4  + \be \s_2^2= 0,$$ 
where $\s_i$ is as usual the $i$th elementary symmetric function and $ a : = (a_1, \dots, a_4)$.

But the main result of this section is negative in this direction:

\begin{theo}\label{symmetric}
Quartic surfaces admitting $\mathfrak S_4$-symmetry form a $4$-dimensional projective space $\sP$
and the general such quartic $X$ is smooth. 

The singular quartics in $\sP$ are contained in four irreducible subvarieties, labeled by the number of singular 
points of the general element in it:
  
  \begin{itemize}
  \item
  $\sP(6)$,  defined by $\beta=0$:
  
   $X \in \sP(6)$ has at least $6$  singular points (the  $\mathfrak S_4$-orbit of $(0,0,1,1)$ is contained in $Sing(X)$ if and only if $\be =0$) and either
   
  $6$ singular points, or
  
   $10$ singular points, the $4$ extra singular points being either the $\mathfrak S_4$-orbit of the point $(0,0,0,1)$, or the $\mathfrak S_4$-orbit of a point $(1,1,1,b)$, or
   
 infinitely many  singular points.
 \item
  $\sP(1)$,  defined by $a_4=0$:
  
   $X \in \sP(1)$ if and only if the point $(1,1,1,1)$ is a  singular point of $X$.
  \item
   $\sP(12)$,  defined by $a_2 a_4 + a_3^2=0$: 
   
   the general $X \in \sP(12)$ has  $12$  singular points (the  $\mathfrak S_4$-orbit of a point $(1,1,b,c)$ with $b\neq c$ and $ b,c \neq 0,1)$
   \item
    $\sP(4)$,  defined by $a_2 ( \be + a_2 + a_3) + a_1 a_4 =0$:
    
     the general $X \in \sP(4)$ has  $4$  singular points (the  $\mathfrak S_4$-orbit of a point $(1,1,1,b)$ with   $ b \neq 1)$.
 \end{itemize}

If we restrict to the subset $\sF$ of normal surfaces inside $\sP$, it turns out that

\begin{enumerate}
\item
 $\sP(12) \cap \sP(m) \cap \sF = \emptyset, \ m=1, 6$, 
 \item
 $X \in \sP(12) \cap \sP(4) \cap \sF $ implies that $X$ has 4 singular points,
 \item
$X \in \sF$ implies that $ |Sing(X)| \leq 12$,
\item
$\sP(1) \cap \sP(6) \cap \sF = \emptyset$,
\item
$\sP(1) \cap \sP(4) \cap \sF \neq  \emptyset$ 
\item
$\sP(6) \cap \sP(4) \cap \sF \neq  \emptyset$ 
\end{enumerate}
 Moreover, a normal symmetric quartic has exactly one of the following cardinalities 
$$1,4,5, 6, 10, 12$$
for its number of singular points.
\end{theo}
 
 We shall prove the theorem through a sequence of auxiliary and more precise results.

We observe preliminarly that the singular set of $X$ is the set 
$$Sing(X) : = \{ x | F(x) = F_i (x) =0 , \ \forall 
1 \leq i \leq 4\},$$
which is clearly a union of $\mathfrak S_4$ orbits.

We have the following main result:

\begin{prop}\label{main-symm}
If we have a singular point of a normal symmetric quartic surface $X$, then the four coordinates cannot be all different 
from each other.

For the  other points, they are singular for a symmetric quartic $X$ according to the following rules:

\begin{itemize}
\item
$\sP(6)$: $(0,0,1,1) \in Sing(X)$ if and only if $\be=0$;
\item
$\sP(4,0,0,0)$: $(0,0,0,1) \in Sing(X)$ if and only if $a_1=a_2=0$;
\item
$\sP(4,b)$: $(1,1,1,b) \in Sing(X)$, $b \neq 1$ if and only if $$a_4= a_2 (1+b)^2 ,  \be= a_1(1+b)^2 +a_2 + a_3;$$
\item
$\sP(1)$: $(1,1,1,1) \in Sing(X)$ if and only if $a_4=0$;
\item
$\sP(12,0,0)$: $(0,0,1,b) \in Sing(X)$, $b \neq 0,1$ if and only if 
$$a_2=a_3=0, a_1 (1+b)^4 + \be b^2=0;$$
\item
$\sP(12,0,1)$: $(0,1,1,z) \in Sing(X)$, $z \neq 0,1$ if and only if 
$$a_3= z a_2, a_4 = z^2 a_2, \be= a_1 z^4 + a_2 z^2 (1+z);$$
\item
$\sP(12,1,1)$: $(1,1,b,c) \in Sing(X)$, $b,c \neq 0,1$, $b \neq c$,  if and only if,
setting $ z : = b+c \neq 0$, 
$$a_3= z a_2, a_4 = z^2 a_2, \be (1 + bc)^2  + a_1 z^4 + a_2 z^2 (1+z)=0.$$
\end{itemize}
\end{prop}

The above Proposition \ref{main-symm} shall be proven through a sequence of lemmas which take care of the 
several cases, and then we shall end the proof giving the final argument.

Let us begin with a  calculation of the partial derivatives, which  yields:
$$ F_i : = \frac{\partial F}{\partial x_i} = a_2 \s_1^2 ( \s_1 + x_i) + a_3 (  \s_3 + \s_1 (\s_2 + \s_1 x_i + x_i^2)) + a_4  \frac{\s_4}{x_i}.$$

In particular, the coordinates of the singular points, since  the following equations are satisfied:
$$ 0 = F_i  x_i  =  a_2  x_i \s_1^2 ( \s_1 + x_i) + a_3 x_i (  \s_3 + \s_1 (\s_2 + \s_1 x_i + x_i^2)) + a_4  \s_4,$$
are roots of the equation
$$ f(z) : = z^3 (a_3 \s_1) + z^2  (a_2 + a_3)  \s_1^2 + z ( a_2 \s_1^3 + a_3 (  \s_3 + \s_1 \s_2 ) )+ a_4  \s_4 = 0 .$$
This gives an idea for  the first assertion of Proposition \ref{main-symm}: because if this equation is not identically zero,
then the four coordinates of a singular point cannot be all different, and if the equation is identically zero,
 we shall see in (III) of the following lemma that the  three exceptional cases  have at least two equal coordinates.

\begin{lemma}\label{sing}
(I) If $a_3 = a_4 =0$ then $Sing(X)$ is infinite, since it contains $\{ x| \s_1=\s_2=0\}$.

(II) If a singular point of $X$ has one coordinate equal to zero,  and $\s_1=0$,
then it has two coordinates equal to zero.

(III) If $X$ is normal, the equation $f(z)=0$ is not identically zero for the points of $Sing(X)$, unless $a_4=0$
and we have the singular point $(1,1,1,1)$,  or unless we have $\be=0$ and we have the singular point $(0,0,1,1)$,
or unless we have $a_2=a_3=0$ and a singular  point with  two coordinates equal to zero.

In particular a singular point never  has four different coordinates.

(IV) If a singular point of $X$ has two coordinates equal to zero,  say $x_1=x_2=0$, then  the equation  $ a_2 \s_1^3  +  a_3 (  \s_3 + \s_1 \s_2)=0$ must be satisfied by the singular point.

Once this equation is satisfied, then for $i=3,4$ 
$$F_i=0 \Leftrightarrow \s_1 x_i ((a_2  + a_3 )   + a_3 x_i ) =0.$$

These two equations are satisfied for

i)  $\s_1=0$, equivalently $x_3  + x_4=0$,  and we get the point $(0,0,1,1)$ (and its $\mathfrak S_4$-orbit). 

ii)  $a_2=0$ and we get then  the  point
$(0,0,0,1)$ (and its $\mathfrak S_4$-orbit).

iii) $a_2=a_3=0$, $x_3, x_4 \neq0$, and then, if we assume  $\s_1 \neq 0$,
it must be 
 $a_1 (x_3 + x_4)^4 + \be (x_3 x_4)^2=0$.

\end{lemma}

\begin{proof}
(I) : the equation is then of the form $a_1 \s_1^4 + a_2 \s_1^2 \s_2 + \be \s_2^2=0$.

(II) If $x_1= \s_1=0$, then $a_3 \s_3 + a_4 x_2 x_3 x_4 =0$, equivalently $(a_3 + a_4)\s_3 =0$,
hence either $\s_3=0$, that is, two coordinates are equal to zero or, since $f(x_i)=0$, 
$a_3 \s_3=0$ and $a_3=a_4=0$, as in (I).

(III) If $f(z)$ is identically zero, then $a_4 \s_4=0$ and either $\s_1= a_3 \s_3= 0$, or $a_2=a_3=0$.
Since $a_3=a_4=0$ is excluded by (I), we get the three cases:

(i) $  \s_4= \s_1= a_3 \s_3= 0$,

(ii)  $a_4 = \s_1=  \s_3= 0$

(iii) $  \s_4= a_2=a_3=0$.

(i):  $\s_4=0$  implies by (II) that two coordinates are zero, and since $\s_1=0$  we have the point $(0,0,1,1)$.
In this case all the symmetric functions vanish except $\s_2=1$, hence it must be $\be=0$.

(ii): then  $\be \s_2^2=0$, hence ($\be=a_4=0$ implies that $\s_1$ divides $F(x)$, a contradiction),
we have $\s_1=\s_2=\s_3=0$, and the singular point must be the point $(1,1,1,1)$.

(iii): since   $\s_4=0$ we may assume $x_1=0$,
and then $F_1= a_4 x_2 x_3 x_4 =0$ implies that  there are two coordinates equal to zero.

(IV) If a singular point of $X$ has two coordinates equal to zero,  say $x_1=x_2=0$, then $a_4  \frac{\s_4}{x_i}$ vanishes for all $i$, and for $i=1,2$  we get that  the equation  $ a_2 \s_1^3  +  a_3 (  \s_3 + \s_1 \s_2)=0$ must be satisfied by the singular point.

Once this equation is satisfied, then for $i=3,4$ 
$$F_i=0 \Leftrightarrow \s_1 x_i ((a_2  + a_3 ) \s_1  + a_3 x_i ) =0.$$

These two equations are satisfied for

i)  $\s_1=0$, or for 

ii)  $x_3=0$, $x_4 \neq 0$, $a_2=0$, or for

iii) $x_3, x_4 \neq0$, $$   a_3 x_i + ( a_3 + a_2) \s_1   =0 \Rightarrow  a_3 ( x_3 +  x_4) =0.$$

For $\s_1=0$, equivalently $x_3  + x_4=0$,  we get the point $(0,0,1,1)$ (and its $\mathfrak S_4$-orbit).

For $x_3=0$, $x_4 \neq 0$, the only possibility, because of   $a_2 x_4=0$, is that  we get the  point
$(0,0,0,1)$ (and its $\mathfrak S_4$-orbit) and  $a_2=0$.

If $a_3=0$ and $\s_1 \neq 0$,  then   $a_2=0$ 
and  $a_1 (x_3 + x_4)^4 + \be (x_3 x_4)^2=0$.

\end{proof}

\begin{lemma}\label{0-1}
The quartic $X_a : = \{ x | F(a,x)=0\}$ has the property that $Sing(X)$ contains the $\mathfrak S_4$-orbit
of the point $(0,0,1,1)$ if and only if $\be=0$, and it contains the $\mathfrak S_4$-orbit
of the point $(0,0,0,1)$ if and only if $a_1= a_2= 0$.
\end{lemma} 
\begin{proof}

We calculate more generally, for later use:
$$ (Sym) \  \s_1 (b,c,1,1) = \s_3 (b,c,1,1) = b+c, \ \s_2 (b,c,1,1) = 1 + bc, \  \s_4 (b,c,1,1) = bc.$$

For $b=c=0$ we get that all $\s_i$ vanish except $\s_2 =1$, hence $(0,0,1,1) \in X$
if and only if $\be=0$; and then, we have a singular point by (III) of Lemma \ref{sing}.

For the point $(0,0,0,1)$   all $\s_i$ vanish except $\s_1 =1$, hence this point is in $X_a$ if and only if  $a_1 = 0$,
and we apply (IV)  of Lemma \ref{sing} to infer that we have a singular point if and only if $a_1= a_2=0$.

\end{proof}

\begin{lemma}\label{onezero}
The quartic $X : = \{ x | F(a_1,a_2,a_3,a_4, \be,x)=0\}$   has the property that $Sing(X)$ contains the $\mathfrak S_4$-orbit
of a point $(0,1,x_3,x_4)$, with $x_3, x_4 \neq0$,  if and only if  this orbit is either

i)  the 
 $\mathfrak S_4$-orbit of a point of the form $(0,1,1,z ), z \neq 0,1$ (consisting of 12 points) and  this holds if and only if
 $$a_3 = a_2 z, a_4 = a_2 z^2, \be = a_1 z^4 +  a_2 z^2 (  1+z),$$
 or 
 
 ii)  $z=1$ (in this case the orbit has 4 points), and this holds  if and only if
 $$ a_2 =  a_4  , \ \be = a_1  +  a_2  + a_3.$$
 
 In particular, we must have $a_2 \neq0$ if $ z\neq 1$.
 \end{lemma}
\begin{proof}

We know from (III) of Lemma \ref{sing} that the four coordinates cannot be all distinct, 
hence our singular point must be of the form $(0,1,1,z )$, and then $\s_1 = z, \ \s_2 = 1, \ \s_3 = z, 
\s_4 = 0$.

We look at the equations derived from the vanishing of the partial derivatives, $F_j=0$.

For $j=1$ we know that 
$$a_2 \s_1^3  + a_3 (  \s_3 + \s_1 \s_2 )) + a_4  \s_3=0,$$
hence for $j=2,3,4$
$$(a_2 + a_3)  \s_1^2   x_i + a_3 (   \s_1  x_i^2) + a_4  \s_3=0,$$
which can be rewritten (since $ z \neq0$) as
$$(a_2 + a_3)  z   x_i + a_3 (     x_i^2) + a_4  =0,$$
This is an equation of degree $2$, and since it is not  identically zero,
by (I) of Lemma \ref{sing}, it has at most two roots.

Since we want $z \neq0$, the conditions that the point is in $X$, plus that we have a singular point (hence $1,z$ are roots of the quadratic equation) are:
$$\be + a_1 z^4 + (a_2 + a_3) z^2=0 ,  \ (  a_2 + a_3 ) z = a_3 +  a_4 ,  \ a_2 z^2 =  a_4 ,$$
for $z=1$ the second equation is a consequence of the third one.  

If $z \neq 1$, then plugging  the third equation in the second and dividing by $(1+z)$ yields $a_3= a_2 z^2$.

\end{proof}

Consider next a point of the form $(b,c,1,1)$,  with $ b \neq c $, and with $ b,c \neq 0, 1$.
Its orbit consists of $12$ points.

\begin{prop}\label{inf}
The quartic 
$$ X_{a,\be} = \{ x| a_1 \s_1^4 + a_2 \s_1^2 \s_2 +  a_3 \s_1 \s_3 + a_4  \s_4  + \be \s_2^2 = 0\} $$
contains the $\mathfrak S_4$-orbit of  the point $(b,c,1,1)$  with $ b \neq c $, and with $ b,c \neq 0, 1$ if and only if, setting $ z: = (b + c) \neq 0$,
the coefficients satisfy 
$$ a_3 = z a_2 , \ a_4 =z^2 a_2, \  \ a_1 z^4 + a_2 z^2 (1+z) + \be ( 1 + bc)^2=0.$$
In particular, if these conditions are satisfied, and moreover $\be=0$, then $X$ has infinitely many singular points.

\end{prop}
\begin{proof}
By the previous Lemma  \ref{0-1},  $ Sing(X)$ contains the $\mathfrak S_4$-orbit
of the point $(0,0,1,1)$ if and only if $\be=0$.

 We are going to see first when the point $(b,c,1,1)$ is a point of $X_{a,\be}$, and then
 when it is a singular point.

 First of all we get a point of $X_{a,\be}$,  by formula {\em (Sym)}, if and only if
 $$a_1 z^{4}   + a_2 z ^{2}  (1  + bc)  +   a_3 z ^2 +  a_4 bc + \be (1  + bc)^2   = 0.$$

For the partial derivatives, at the point  $(b,c,1,1)$, since
 the condition for the singular points with all coordinates different from zero
boils down, for the given point, to $1,b,c$ being roots 
  of the equation
$$ f(w) : = w^3 (a_3 \s_1) + w^2  (a_2 + a_3)  \s_1^2 + w ( a_2 \s_1^3 + a_3 (  \s_3 + \s_1 \s_2 ) )+ a_4  \s_4 = 0 ,$$
equivalently  of the equation
 $$ f(w) : = w^3 (a_3 z ) + w^2  (a_2 + a_3)  z^2 + w ( a_2 z^3 + a_3 (  z   bc  ))+ a_4  bc = 0 .$$
 
 Since $f(w) = (a_3 z ) (w+1)(w+b)(w+c)$,  it must be $a_4 , a_ 3 \neq 0$, and then $ a_4 = a_3 z$, and dividing by $a_4$, 
 $$ 1 + z =  z (a_2/a_3  + 1) \Leftrightarrow a_3 = a_2 z,$$
 and then the third claimed equality  holds automatically, since we are then left with the requirement that $  bc =   bc $.
 
We can then rewrite the condition that the point lies in $X_{a,\be}$ as

 $$a_1 z^{4}   + a_2 z ^{2}    +   a_2 z ^3  + \be (1  + bc)^2   = 0.$$
 
 To finish the proof, we observe that if $\be=0$, 
 since the equations depend only  on $z  $, we get infinitely many singular points varying $b,c$ with $b + c= z$.

\end{proof}

{\em End of the proof of Proposition \ref{main-symm}.}

\medskip

 The singular points with some  coordinate equal to zero have been considered in the previous Lemma \ref{sing},
 Lemma \ref{0-1}, Lemma \ref{onezero}, hence we are only left with the case where all coordinates
 are non zero, but only two values are achieved (if the coordinates take three distinct values, we obtain the situation of the previous proposition \ref{inf}).

Therefore  only two possibilities remain. 
\bigskip 

If the singular point is of the form $(1,1,b,b)$ then $\s_1=\s_3 = 0$, $\s_2 = 1 + b^2$, $\s_4 = b^2$.

From the equation $f(w)=0$, since $\s_1 = \s_3=0$, we infer $a_4 \s_4 =0$, and since we assume $b \neq 0$,
we obtain $a_4=0$. Then $ 0 = F(1,1,b,b) = \be (1 + b^2)^2$ implies $\be=0$, and then $X$ is reducible, ($F$ is divisible by $\s_1$), or $b=1$, and this is a singular point of $X$ if and only if $a_4=0$.

\bigskip

If instead the point is $(1,1,1,b)$ it follows that $\s_1= 1 + b, \s_4 = b, \s_2 =  1+b, \s_3 = 1+b.$

Since the case  $b=1$ was  already treated, we assume that $b \neq 1$.

The condition that $1,b$ are roots of the cubic equation $f(w)=0$  is easily seen to be equivalent to the single condition 
$$ a_2 (1+b)^2 = a_4.$$

The condition that $(1,1,1,b) \in X$ boils then down to
$$a_1 (1+b)^4 +   a_2 (1+b)^3 + a_3 (1+b)^2  + a_4 b + \be  (1+b)^2 =0,$$
which, after using $ a_2 (1+b)^2 = a_4$ and after dividing by $(1+b)^2$ boils down to
$$ \be + a_3 + a_2 +  a_1 (1+b)^2 =0.$$
With the customary notation $ z : = (1+b)$, we get 
$$a_4 = a_2 z^2 ,  \be + a_3 + a_2 +  a_1 z ^2 =0.$$

\qed

\begin{rem}
In the above equation  $a_4 = a_2 (1+b)^2$,  since $(1+b) \neq 0, a_4 \neq 0$, it cannot be $a_2=0$.

It follows then, since $a_2 \neq 0$, that  $b$ is uniquely determined by this equation.
\end{rem}

\begin{prop}\label{ten}
In the pencil of quartics $$ X_c: = \{ c \s_1 \s_3 + \s_4 =0\}$$
the quartic has always $10$ singular points except for $c=0$.

The singularities are nodes ($A_1$-singularities).

\end{prop}
\begin{proof}
By Lemma  \ref{sing} the only singular points with at least two coordinates equal to zero
are just the orbits of $(0,0,1,1)$ and $(0,0,0,1)$ if $a_3 = c\neq 0$.

Points with just one coordinate equal to zero are excluded by lemma \ref{onezero},
while singular points with nonzero coordinates taking three values are excluded by
proposition \ref{inf}. 

For points of type $(1,1,b,b)$ $\s_1=0$  hence they cannot lie in $ X_c$, since $a_4=1$.

For points of type $(1,1,1,b)$, $b\neq 1$, they lie in $ X_c$ iff $ c (1+b)^2 = b$,
but as we saw the condition that we have a singular point boils down to
$$ a_2 (1+b)^2 = a_4=1,$$ impossible since $a_2=0$.

For the last assertion, at the point $(0,0,1,1)$ we have $x_3=1$ and local coordinates $x_1, x_2, \s_1$: 
then the quadratic part of the equation is
$$ x_1 x_2 + \s_1 (x_1 +x_2)$$
 and we have a node.

Likewise, at the point $(0,0,0,1)$ we have $x_4=1$ and local coordinates $x_1, x_2, x_3$:
then the quadratic part of the equation is
$$ \s_3 = x_1 x_2 +x_1 x_3 + x_2 x_3 $$
and again we have a node.

\end{proof}

\bigskip

{\em Proof of Theorem \ref{symmetric}}

We have already seen in Proposition \ref{main-symm} that the only singular orbit with 6 elements is the
orbit of the point $(0,0,1,1)$, and this point is a singular point if and only if $\be=0$,
that is, $ X \in \sP(6)$. Similarly the only orbit with one element is the point $(1,1,1,1)$, which is singular if and only if
$a_4=0$, that is, $X \in \sP(1)$. 

We prove directly now that $ \sP(1) \cap  \sP(6)$ contains no normal surface: since $\be=a_4=0$ 
imply that $\s_1 $ divides the equation $F$ of $X$.

We pass now to consider $\sP(12)$, the closure of the locus of quartics with an orbit of singular points
having 12 elements. This locus consists of three sets, $\sP(12,0,0) , \sP(12,0,1),  \sP(12,1,1)$,
and for the second and third set there must exist $z \neq 0$ such that 
$ a_3= z a_2 , a_4 = z^2 a_2$.  In particular we must have $a_2 a_4 + a_3^2=0$.

The above equation holds in particular if $a_2= a_3=0$, and then we can find a `unique' $b \neq 0,1$
such $a_1 (1+b^4) + \be b^2=0$ provided  $a_1 \neq 0, \be \neq 0$ (in fact, the two roots $b, \frac{1}{b}$
yield the same $\mathfrak S_4$-orbit  in projective space).

If instead $a_2 a_4 + a_3^2=0$ and $a_2, a_3 \neq 0$, we find $z \neq 0 $ such 
$ a_3= z a_2 , a_4 = z^2 a_2$, and then the pair $b,c$ is determined by 
the conditions that $b+c = z$,
and that  $bc$ is the solution of the equation 
$\be (1 + bc)^2  + a_1 z^4 + a_2 z^2 (1+z)=0.$
Of course $bc=0$ if and only if we are in case $\sP(12,0,1)$.

The locus  $\sP(4)$,  defined by $a_2 ( \be + a_2 + a_3) + a_1 a_4 =0$, clearly contains the 
loci $\sP(4,0,0,0)$, $\sP(4,b)$: moreover if the above equation is satisfied, we can find a unique $b$
(equivalently, a unique $(1+b)^2$) such that 
$a_4= a_2 (1+b)^2 ,  \be= a_1(1+b)^2 +a_2 + a_3,$
if $a_1 \neq 0$ or $a_2 \neq 0$
(observe that the  equation of $\sP(4)$ is the condition for the  simultaneous solvability of both equations
for $(1+b)^2$).

We pass now to further intersection properties of these loci.

We have seen that $\sP(6) \cap \sP(1)$ contains no normal surface, while 
$\sP(6) \cap \sP(4)$ contains normal surfaces by Proposition \ref{ten}.

 $\sP(4) \cap \sP(1)$ is the union of $a_2=a_4=0$ and of $a_4=\be + a_2 + a_3 =0$.
 These are two components whose general element has 5 singular points.
 
 That $\sP(6) \cap \sP(12)$ contains no normal surface follows since if $\be=0$
 the locus $\sP(6) \cap \sP(12,0,0)$ consists of the surface $\{\s_4=0\}$, while 
  the locus $ \sP(6) \cap (\sP(12,0,1) \cup \sP(12,1,1))$ consists of surfaces with infinitely many singular points
  (for each choice of $b,c$ with $ b+c = z$ we get a singular point).
  
  That $\sP(1) \cap \sP(12)$ contains no normal surface follows by (I) of Lemma \ref{sing},
  since then $a_3=a_4=0$.
  
  We consider now $\sP(4) \cap \sP(12)$, defined by 
  $$a_2 ( \be + a_2 + a_3) + a_1 a_4 =0, \ a_2 a_4 + a_3^2=0.$$
  First of all, if $a_2=0$, then $a_3= a_1 a_4=0$, and $a_2=a_3=a_4=0$
  yields surfaces of the form $a_1 \s_1^4 + \be \s_2^2$, hence with singular curve
  $ \s_1 = \s_2=0$.
  
  Instead, the case  $a_1= a_2=a_3=0$ yields a surface of the form (if irreducible)
  $X = \{\s_4 + \be \s_2^2=0\}$.
  
  At the four coordinate points both $\s_2, \s_4$ vanish, and we see easily that these
  are uniplanar double points ( if $x_1=1, x_j =0, j \geq 2$, then the local equation is
  $$ (x_2 + x_3 + x_4)^2 + x_2 x_3 x_4 + (x_2 x_3 + x_2  x_4+  x_3 x_4)^2=0).$$
  An easy inspection of the cases of Proposition \ref{main-symm} shows that there
  are no other singular points for $\be \neq 0$.
  
  We may now assume that $a_2=1$, hence $a_3=z \neq 0$, $a_4=z^2$, 
  and
  $$ (*) \ a_1 z^2 + (1+z+ \be)=0.$$
  
  The point $(1,1,1,1+z)$ and its orbit are then  singular points of $X$. 
  We claim that these are all.
  
  In fact, there is no orbit consisting of 12 singular points by the following arguments.
  
  Case $ \sP(12,0,0)$ is excluded since we have $ a_2 \neq 0$.
  
  In case $ \sP(12,0,1)$ $(*)$ and $ \be = a_1 z^4 + z^2 (1+z)$ imply
  $ \be = \be z^2$, hence $\be =0$, since  $z \neq 1$. We are then done since we have shown 
  $\sP(6) \cap \sP(12) \cap \sF = \emptyset.$
  
  In case $ \sP(12,1,1)$ $(*)$ and $ \be (1 + bc)^2  = a_1 z^4 + z^2 (1+z)$ imply
  $ \be (1 + z^2 + (bc)^2)=0$. Since $\be =0$ leads to a contradiction as above, 
  we have
  $$ 0 = 1 + b + c + bc = (1+b)(1+c),$$
  a contradiction since $b,c \neq 0,1$.
  
  Likewise, there is no other orbit of singular points with cardinality $4,6,1$ by Proposition \ref{main-symm}. 
  
  To finish the proof that there are no more than 12 singular points, it suffices to observe that
  cases $ \sP(12,0,0)$, $ \sP(12,0,1)$, $ \sP(12,1,1)$ are mutually exclusive for $\be \neq0$,
  since $a_2=0$ in the first case, and $a_2 \neq 0$ for the other two, while
  the second and third case are exclusive because $bc \neq 0$.
  
  Finally, one sees also that the two subcases of $\sP(4)$ are mutually exclusive,
  hence the possible cardinalities of $Sing(X)$, for $X$ a normal symmetric surface,
  are only $1,4,5,6,10,12$.

\qed 

\medskip

 We summarize in the next corollary some  result obtained so far:

\begin{cor}
The maximal number of singular points that  a normal  symmetric quartic  $X$  can have
is exactly  $12$.

The  symmetric quartics of the form $X = \{\s_4 + \be \s_2^2=0\} , \be \neq 0$
have 4 uniplanar double points, which are  singularities of type $D_4$.

\end{cor}

\begin{proof}
We have the local equation
$$ (x_2 + x_3 + x_4)^2 + x_2 x_3 x_4 + (x_2 x_3 + x_2  x_4+  x_3 x_4)^2=0.$$
Blowing up the singular point, we obtain a line in the exceptional $\PP^2$,
with three singular points which are nodes,
since we get the equation
$$ x^2 + u y z (x+y+z) + u^2 (\dots)=0,$$
where $u=0$ is the equation of the exceptional divisor.

\end{proof}

\bigskip

\bigskip

{\bf Acknowledgement:} I would like to thank Stephen Coughlan,  the referee of \cite{kummers},
  and especially Matthias Sch\"utt for interesting remarks, helpful comments and suggestions, and encouragement.
  
  The construction in Subsection 2.3 is due to Matthias Sch\"utt.
  
  Thanks to a referee for pointing out some inaccuracies.

\end{document}